\documentclass[a4paper,11 pt,reqno]{amsart}
\usepackage{a4,amsbsy,amscd,amssymb,amstext,amsmath,mathtools,mathdots,latexsym,oldgerm,enumerate,verbatim,graphicx,xy,ytableau}
\makeatletter                                       %De subsecties in de inhoudsopgave moeten
\def\l@subsection{\@tocline{2}{0pt}{2.5pc}{2.5pc}{}}%iets opgeschoven worden naar rechts
\makeatother                                        %een eventuele tweede regel moet gewoon
\makeatletter                                                  %Hoofdstukken moeten op de eerst
\def\chapter{\clearpage\thispagestyle{plain}\global\@topnum\z@ %beschikbare pagina beginnen
\@afterindenttrue \secdef\@chapter\@schapter}
\makeatother

\newtheorem{thmgl} {Theorem}    %globaal genummerd

\newtheorem{lemgl} {Lemma}
\theoremstyle{definition}

\newtheorem{remgl} {Remark}
\newtheorem{remsgl} [remgl]{Remarks}

\newcommand{\mf}{\mathfrak}
\newcommand{\mc}{\mathcal}

\newcommand{\nts}{\negthinspace}     %handig
\newcommand{\Nts}{\nts\nts}

\newcommand{\ov}{\overline}

\newcommand{\ot}{\otimes}           %vectorruimten en modules

\newcommand{\Hom}{{\rm Hom}}        %Algebra algemeen
\newcommand{\Mor}{{\rm Mor}}

\newcommand{\Der}{{\rm Der}}

\newcommand{\ind}{{\rm ind}}

\newcommand{\Ker}{{\rm Ker}}
\renewcommand{\Im}{{\rm Im}}
\newcommand{\codim}{{\rm codim}}

\newcommand{\tr}{{\rm tr}}

\newcommand{\g}{\mf{g}}
\newcommand{\h}{\mf{h}}
\let\ttie\t
\newcommand{\tie}[1]{{\let\t\ttie \ttie#1}}%\t requires a special treatment, because
\renewcommand{\t}{\mf{t}}  %it is defined recursively. This trick is due to Uwe L\"uck

\newcommand{\m}{\mf{m}}
\renewcommand{\u}{\mf{u}}

\renewcommand{\b}{\mf{b}}
\newcommand{\Ad}{{\rm Ad}}              %Algebraische groepen
\newcommand{\Lie}{{\rm Lie}}
\newcommand{\Dist}{{\rm Dist}}
\newcommand{\GL}{{\rm GL}}

\newcommand{\SL}{{\rm SL}}

\newcommand{\SO}{{\rm SO}}

\newcommand{\Sp}{{\rm Sp}} %symplectische groep

\newcommand{\ve}{\varepsilon}

\makeatletter
\def\vcdots{\vbox{\baselineskip4\p@ \lineskiplimit\z@
\kern3\p@\hbox{.}\hbox{.}\hbox{.}\Nts\nts\kern3\p@}}
\makeatother

\xyoption{all}

\begin{document}

\title{On the first restricted cohomology of a reductive Lie algebra and its Borel subalgebras}

\begin{abstract}
Let $k$ be an algebraically closed field of characteristic $p>0$ and let $G$ be a connected reductive group over $k$.
Let $B$ be a Borel subgroup of $G$ and let $\g$ and $\b$ be the Lie algebras of $G$ and $B$.
Denote the first Frobenius kernels of $G$ and $B$ by $G_1$ and $B_1$.
Furthermore, denote the algebras of regular functions on $G$ and $\g$ by $k[G]$ and $k[\g]$, and similar for $B$ and $\b$.
The group $G$ acts on $k[G]$ via the conjugation action and on $k[\g]$ via the adjoint action.
Similarly, $B$ acts on $k[B]$ via the conjugation action and on $k[\b]$ via the adjoint action.
We show that, under certain mild assumptions, the cohomology groups $H^1(G_1,k[\g])$, $H^1(B_1,k[\b])$, $H^1(G_1,k[G])$ and
$H^1(B_1,k[B])$ are zero. We also extend all our results to the cohomology for the higher Frobenius kernels.
\end{abstract}

\author[R.\ Tange]{Rudolf Tange}

\maketitle
\markright{\MakeUppercase{First restricted cohomology of a reductive Lie algebra}}

\section*{Introduction}
Let $k$ be an algebraically closed field of characteristic $p>0$, let $G$ be a connected reductive group over $k$,
and let $\g$ be the Lie algebra of $G$. Recall that $\g$ is a restricted Lie algebra: it has a $p$-th power map
$x\mapsto x^{[p]}:\g\to\g$, see \cite[I.3.1]{Bo}. In the case of $G=\GL_n$ this is just the $p$-th matrix power.
A $\g$-module $M$ is called {\it restricted} if $(x^{[p]})_M=(x_M)^p$ for all $x\in M$. Here $x_M$ is the endomorphism
of $M$ representing $x$.

Recall that an element $v$ of a $\g$-module $M$ is called a $\g$-invariant
if $x\cdot v=0$ for all $x\in\g$. We denote the space of $\g$-invariants in $M$ by $M^\g$. The right derived functors
of the left exact functor $M\mapsto M^\g$ from the category of restricted $\g$-modules to the category of vector spaces over $k$
are denoted by $H^i(G_1,-)$.

Let $k[\g]$ be the algebra of polynomial functions on $\g$.
If one is interested in describing the algebra of invariants $(k[\g]/I)^\g$ for some ideal $I$ of $k[\g]$,
then it is of interest to know if $H^1(G_1,k[\g])=0$, because then we have an exact sequence
$$k[\g]^\g\to (k[\g]/I)^\g \to H^1(G_1,I)\to 0$$
by the long exact cohomology sequence.
So, in this case, $(k[\g]/I)^\g$ is built up from the image of $k[\g]^\g$ in $k[\g]/I$, and $H^1(G_1,I)$.

The paper is organised as follows. In Section~\ref{s.prelim} we state some results from the literature
that we need to prove our main result. This includes a description of the algebra of invariants $k[\g]^G$,
the normality of the nilpotent cone $\mc N$, and some lemmas on graded modules over graded rings.
In Section~\ref{s.H^1(G_1,k[g])} we prove Theorems~\ref{thm.H^1(G_1,k[g])} and \ref{thm.H^1(B_1,k[b])} which
state that, under certain mild assumptions on $p$, $H^1(G_1,k[\g])$ and $H^1(B_1,k[\b])$ are zero.
In Section~\ref{s.H^1(G_1,k[G])} we prove Theorems~\ref{thm.H^1(G_1,k[G])} and \ref{thm.H^1(B_1,k[B])}
which state that, under certain mild assumptions on $p$, $H^1(G_1,k[G])$ and $H^1(B_1,k[B])$ are zero.
In Section~\ref{s.higherFrobenius} we extend Theorems~\ref{thm.H^1(G_1,k[g])}-\ref{thm.H^1(B_1,k[B])}
to the cohomology for the higher Frobenius kernels $G_r$ and $B_r$, $r\ge2$.

We briefly indicate some background to our results. For convenience we only discuss the $G$-module $k[\g]$.
As is well-known, under certain mild assumptions $k[\g]$ has a good filtration, see \cite{Don} or \cite[II.4.22]{Jan}.
So a natural first approach to prove that $H^1(G_1,k[\g])=0$ would be that to show that
$H^1(G_1,\nabla(\lambda))=0$ for all induced modules $\nabla(\lambda)$ that show up in a good filtration
of $k[\g]$. However, this isn't true: even for $p>h$, $h$ the Coxeter number, one can easily deduce counterexamples
from \cite[Cor.~5.5]{AJ} (or \cite[II.12.15]{Jan}).\footnote{This approach
{\it does} work when proving the (well-known) result that $H^1(G,k[\g])=0$.}
\begin{comment}
IN DETAIL:\\
By \cite[II.12.15]{Jan} we have for $p>h$, $\mu=s_{\alpha}\cdot 0+p\lambda=p\lambda-\alpha$ dominant and $\alpha$ simple (then $\lambda$ is also dominant)
that $H^1(G_1,H^0(\mu))^{[-1]}\cong\ind_B^G\lambda$.
For $G=\SL_n$ and $p>n$ we could take $\lambda=\ve_1-\ve_n$ or $\lambda=(n-1)\ve_1-\sum_{i=2}^n\ve_i$, and $\alpha=\ve_1-\ve_2$.
In general we could take $\lambda=2p\rho$: $p$ times the sum of the positive roots, and $\alpha$ any simple root.\\
\end{comment}
It is also easy to see that we cannot have
$H^i(G_1,k[\g])=0$ for all $i>0$: the trivial module $k$ is direct summand of $k[\g]$, and for $p>h$
we have $H^{\bullet}(G_1,k)\cong k[\mc N]$ where the degrees of $k[\mc N]$ are doubled, see \cite[II.12.14]{Jan}.

The idea of our proof that $H^1(G_1,k[\g])=0$ is as follows. Noting that $H^1(G_1,k[\g])$ is a $k[\g]^G$-module,
we interpret a certain localisation of $H^1(G_1,k[\g])$ as the cohomology group of the coordinate ring
of the generic fiber of the adjoint quotient map $\g\to\g/\nts/G$.
It is easy to see that this cohomology group has to be zero, so we are left with showing that
$H^1(G_1,k[\g])$ is torsion-free over the invariants $k[\g]^G$. To prove the latter we use Hochschild's characterisation
of the first restricted cohomology group and a ``Nakayama Lemma type result".
The ideas of the proofs of the other main results are completely analogous.

\section{Preliminaries}\label{s.prelim}
Throughout this paper $k$ is an algebraically closed field of characteristic $p>0$.
For the basics of representations of algebraic groups we refer to \cite{Jan}.
\subsection{Restricted representations and restricted cohomology}\label{ss.restrictedcohomology}
Let $G$ be a linear algebraic group over $k$ with Lie algebra $\g$.
Let $G_1$ be the first Frobenius kernel of $G$ (see \cite[Ch.~I.9]{Jan}). It is an infinitesimal group scheme with
$\dim k[G_1]=p^{\dim(\g)}$. Its category of representations is equivalent to the category of 
restricted representations of $\g$, see the introduction.

Let $M$ be an $G_1$-module. By \cite{Hoch} (see also \cite[I.9.19]{Jan}) we have
\begin{equation}\label{eq.restrictedcohomology}
H^1(G_1,M)=\{\text{restricted derivations}:\g\to M\}/\{\text{inner derivations of } M\}.
\end{equation}
Here a {\it derivation} from $\g$ to $M$ is a linear map $D:\g\to M$ satisfying
\begin{equation*}\label{eq.der}
D([x,y])=x\cdot D(y)-y\cdot D(x)
\end{equation*}
for all $x,y\in\g$. Such a derivation is called {\it restricted} if
\begin{equation*}\label{eq.restricted}
D(x^{[p]})=(x_M)^{p-1}(D(x))
\end{equation*}
for all $x\in\g$, 
where $x_M$ is the vector space endomorphism of $M$ given by the action of $x$, and $-^{[p]}$ denotes the $p$-th power map of $\g$.
An {\it inner derivation} of $M$ is a map $x\mapsto x\cdot u:\g\to M$ for some $u\in M$. If $M$ is restricted, then every inner derivation is restricted.
Clearly $H^1(G_1,M)$ is an $G$-module with trivial $\g$-action: If $D$ is a derivation and $y\in\g$, then $[y,D]$ is the inner derivation given by $D(y)$.
Note also that $H^1(G_1,k[\g])$ is a $k[\g]^\g$-module, since the restricted derivations $\g\to k[\g]$ form a $k[\g]^\g$-module and the map $f\mapsto(x\mapsto x\cdot f)$
from $k[\g]$ to the restricted derivations $\g\to k[\g]$ is $k[\g]^\g$-linear.

\subsection{Actions of restricted Lie algebras}\label{ss.restrictedaction}
Let $\g$ be a restricted Lie algebra over $k$. Following \cite{Skry} we define an action of $\g$ on an affine variety $X$ over $k$
to be a homomorphism $\g\to\Der_k(k[X])$ of restricted Lie algebras, where $\Der_k(k[X])$ is the (restricted) Lie algebra of $k$-linear derivations of $k[X]$.
It is easy to see that this includes the case that $X$ is a restricted $\g$-module.
If $\g$ acts on $X$ and $x\in X$, then we define $\g_x$ to be the stabiliser in $\g$ of the maximal ideal $\m_x$ of $k[X]$ corresponding to $x$.
In case $X$ is a closed subvariety of a restricted $\g$-module, then we have $\g_x=\{y\in\g\,|\,y\cdot x=0\}$.
\begin{lemgl}\label{lem.liealginvs}
Let $\g$ be a restricted Lie algebra over $k$ acting on a normal affine variety $X$ over $k$.
If $\max_{x\in X}\codim_\g\g_x=\dim X$, then $k[X]^\g=k[X]^p$.
\end{lemgl}
\begin{proof}
By \cite[Thm.~5.2(5)]{Skry} we have $[k(X):k(X)^\g]=p^{\dim(X)}$.
By \cite[Cor.~3 to Thm.~V.16.6.4]{Bou0} %take K=k and L=k(X)
we have $[k(X):k(X)^p]=p^{\dim(X)}$. So $k(X)^\g=k(X)^p$, since we always have $\supseteq$.
Clearly, $k(X)^p={\rm Frac}(k[X]^p)$, $k(X)^\g={\rm Frac}(k[X]^\g)$ and
$k[X]^\g$ is integral over $k[X]^p$. Since $X$ is normal variety, $k[X]^p\cong k[X]$ is a normal ring.
It follows that $k[X]^\g=k[X]^p$.
\end{proof}

\subsection{Two lemmas on graded rings and modules}\label{ss.graded}
We recall a version of the graded Nakayama lemma which follows from \cite[Ch.~13]{Pas}, Lem.~4, Ex.~3, Lem.~3.
\begin{lemgl}[{\cite[Ch.~13]{Pas}}]\label{lem.gradednakayama}
Let $R=\bigoplus_{i\ge0}R^i$ be a positively graded ring with $R^0$ a field, let $M$ be a positively graded $R$-module
and let $(x_i)_{i\in I}$ be a family of homogeneous elements of $M$. Put $R^+=\bigoplus_{i>0}R^i$.
\begin{enumerate}[{\rm(i)}]
\item If the images of the $x_i$ in $M/R^+M$ span the vector space $M/R^+M$ over $R^0$, then the $x_i$ generate $M$.
\item If $M$ is projective and the images of the $x_i$ in $M/R^+M$ form an $R^0$-basis of $M/R^+M$, then $(x_i)_{i\in I}$ is an $R$-basis of $M$.
\end{enumerate}
\end{lemgl}

\begin{lemgl}\label{lem.direct_summand}
Let $R$ be a positively graded ring with $R^0$ a field and let $N$ be a positively graded $R$-module which is projective. %So N is graded free.
\begin{enumerate}[{\rm(i)}]
\item Let $M$ be a submodule of $N$ with $(R^+N)\cap M\subseteq R^+M$. Then $M$ is free and a direct summand of $N$.
\item Let $M$ be a positively graded $R$-module, let $\varphi:M\to N$ be a graded $R$-linear map and let $\ov \varphi:M/R^+M\to N/R^+N$ be the induced $R^0$-linear map.
Assume the canonical map $M\to M/R^+M$ maps $\Ker(\varphi)$ onto $\Ker(\ov \varphi)$.
Then $\Im(\varphi)$ is free and a direct summand of $N$.
\end{enumerate}
\end{lemgl}
\begin{proof}
(i).\ From the assumption it is immediate that the natural map $M/R^+M\to N/R^+N$ is injective. Now choose an $R^0$-basis $(\ov x_i)_{i\in I}$ of $M/R^+M$
and extend it to a basis $(\ov x_i)_{i\in I\cup J}$ of $N/R^+N$. Let $(x_i)_{i\in I\cup J}$ be a homogeneous lift of this basis to $N$. Then this is a basis
of $N$ by Lemma~\ref{lem.gradednakayama}(ii). Furthermore, $(x_i)_{i\in I}$ must span $M$ by Lemma~\ref{lem.gradednakayama}(ii).
So $M$ is (graded-) free and has the $R$-span of $(x_i)_{i\in J}$ as a direct complement.\\
(ii).\ By (i) it suffices to show that $(R^+N)\cap\Im(\varphi)\subseteq R^+\Im(\varphi)$. Let $x\in M$ and assume that $\varphi(x)\in R^+N$.
Then $\ov x:=x+R^+M\in\Ker(\ov \varphi)$. By assumption there exists $x_1\in\Ker(\varphi)$ such that $\ov x=\ov x_1$.
Then $x-x_1\in R^+M$ and $\varphi(x)=\varphi(x-x_1)\in R^+\Im(\varphi)$.
\end{proof}
There is also an obvious version of Lemma~\ref{lem.direct_summand} (and of course of Lemma~\ref{lem.gradednakayama}) for a local ring $R$:
simply assume $R$ local, omit the gradings everywhere and replace $R^+$ by the maximal ideal of $R$.

\subsection{The standard hypotheses and consequences}\label{ss.standardhyp}
In the remainder of this paper $G$ is a connected reductive group over $k$ and $\g$ is its Lie algebra.
Recall that $\g$ is a restricted Lie algebra, see \cite[I.3.1]{Bo}, we denote its $p$-th power map by $x\mapsto x^{[p]}$.
The group $G$ acts on $\g$ and the nilpotent cone $\mc N$ via the adjoint action and on $G$ via conjugation,
and therefore it also acts on their algebras of regular functions: $k[\g]$, $k[\mc N]$ and $k[G]$.
Fix a maximal torus $T$ of $G$ and let $\t$ be its Lie algebra.
We fix an $\mathbb F_p$-structure on $G$ for which $T$ is defined and split over $\mathbb F_p$.
Then $\g$ has an $\mathbb F_p$-structure and $\t$ is $\mathbb F_p$-defined. Denote the $\mathbb F_p$-defined regular functions
on $\g$ and $\t$ by $\mathbb F_p[\g]$ and $\mathbb F_p[\t]$.
We will need the following standard hypotheses, see \cite[6.3, 6.4]{Jan2} or \cite[2.6, 2.9]{Jan3}:
\begin{enumerate}[(H1)]
\item The derived group $DG$ of $G$ is simply connected,
\item $p$ is good for $G$,
\item There exists a $G$-invariant non-degenerate bilinear form on $\g$.
\end{enumerate}
Put $G_x=\{g\in G\,|\,\Ad(g)(x)=x\}$ and $\g_x=\{y\in\g\,|\,[y,x]=0\}$.
Assuming (H1)-(H3) we have by \cite[2.9]{Jan3} that $\Lie(G_x)=\g_x$ for all $x\in\g$. See also \cite[Sect.~2.1]{PS}.
Put $n=\dim(T)$. We call $x\in\g$ {\it regular} if $\dim G_x$ (or $\dim\g_x$) equals $n$, the minimal value.
Under assumptions (H1) and (H3) we have that $d\alpha\ne0$ for all roots $\alpha$,
so restriction of functions defines an isomorphism $k[\g]^G\stackrel{\sim}{\to}k[\h]^W$, see \cite[Prop.~7.12]{Jan3}.
The set of regular semisimple elements in $\g$ is the nonzero locus of the regular function $f_{\rm rs}$ on $\g$
which corresponds under the above isomorphism to the product of the differentials of the roots.
Note that $f_{\rm rs}\in\mathbb F_p[\g]$: $f_{\rm rs}$ is defined over $\mathbb F_p$.

Under assumptions (H1)-(H3) it follows from work of Demazure \cite{Dem} that $k[\t]^W$ is a polynomial algebra in
$n$ homogeneous elements defined over $\mathbb F_p$, see \cite[9.6 end of proof]{Jan2}. %It's also mentioned in \cite[7.14]{Ja3}
We denote the corresponding elements of $\mathbb F_p[\g]$ by $s_1,\ldots,s_n$.
Assuming (H1)-(H3) the vanishing ideal of $\mc N$ in $k[\g]$ is generated by the $s_i$, see \cite[7.14]{Jan3}, and
all regular orbit closures are normal, in particular $\mc N$ is normal, see \cite[8.5]{Jan3}.
%Rem: from the proof of \cite[Prop.~2.11]{Jan3} one easily deduces that $f(X)=f(X_s)$ for all $f\in k[\g]^G$
%So the function $f_{\rm rs}$ above with $f_{\rm rs}(X)\ne0\implies X \text{ regular}$ for $X\in\g$ semisimple, satisfies in fact:
%$f_{\rm rs}(X)\ne0\implies X \text{ regular semisimple}$ for all $X\in\g$: if $X_s$ is regular, then $X_n=0$.

We call $g\in G$ {\it regular} if $G_g:=\{h\in G\,|\,hgh^{-1}=g\}$ has dimension $n$, the minimal value.
Restriction of functions defines an isomorphism $k[G]^G\stackrel{\sim}{\to}k[T]^W$, see \cite[6.4]{St}.
The set of regular semisimple elements in $G$ is the nonzero locus of the regular function $f'_{\rm rs}$ on $G$
which corresponds under the above isomorphism to $\prod_{\alpha\text{ a root}}(\alpha-1)$. %See \cite[6.8]{St}
If $G$ is semisimple, simply connected, then $k[G]^G$ is a polynomial algebra in the characters $\chi_1,\ldots,\chi_n$
of the irreducible $G$-modules whose highest weights are the fundamental dominant weights. Furthermore, the schematic
fibers of the adjoint quotient $G\to G/\nts/G$ are reduced and normal and they are regular orbit closures.
See \cite{St} and \cite[4.24]{Hum}. One can also deduce from (H1)-(H3) that $\Lie(G_g)=\g_g:=\{x\in\g\,|\,\Ad(g)(x)=x\}$.

\section{The cohomology groups $H^1(G_1,k[\g])$ and $H^1(B_1,k[\b])$}\label{s.H^1(G_1,k[g])}
Throughout this section we assume that hypotheses (H1)-(H3) from Section~\ref{ss.standardhyp} hold.
\begin{thmgl}\label{thm.H^1(G_1,k[g])}
$H^1(G_1,k[\g])=0$.
\end{thmgl}
\begin{proof}
Let $K$ be an algebraic closure of the field of fractions of $R:=k[\g]^G$.
Since the action of $\g$ on $k[\g]$ is $R$-linear we have $H^1(G_1,k[\g])=H^1\big((G_1)_R,k[\g]\big)$,
where $-_R$ denotes base change from $k$ to $R$, see \cite[I.1.10]{Jan}.
So, by the Universal Coefficient Theorem \cite[Prop.~I.4.18]{Jan}, we have
$$K\ot_RH^1(G_1,k[\g])=H^1\big((G_1)_K,K\ot_Rk[\g]\big)=H^1\big((G_K)_1,K\ot_Rk[\g]\big)\,.$$ %there is no Tor term, since localisation is exact.
For $i\in\{1,\ldots,n\}$ denote the regular function on $\g_K$ corresponding to $s_i\in k[\g]$ by $\tilde s_i$.
Then $K\ot_Rk[\g]=K[\g_K]/(\tilde s_1-s_1,\ldots,\tilde s_n-s_n)=K[F]$, where $F\subseteq\g_K$ is the fiber of the morphism
$$x\mapsto(\tilde s_1(x),\ldots,\tilde s_n(x)):\g_K\to\mathbb A_K^n$$ over the point $(s_1,\ldots,s_n)\in\mathbb A_K^n$.
Let $f_{\rm rs}\in\mathbb F_p[\g]\cap k[\g]^G$ be the polynomial function from Section~\ref{ss.standardhyp} with nonzero locus the set of regular semisimple elements in $\g$,
and let $\tilde f_{\rm rs}$ be the corresponding polynomial function on $\g_K$. Then we have for all $x\in F$ that $\tilde f_{\rm rs}(x)=f_{\rm rs}\ne0$.
So $F$ consists of regular semisimple elements. By \cite[Lem.~3.7, Thm.~3.14]{St2} this means that $F=G_K/S$ for some maximal torus $S$ of $G_K$.
In particular, $K[F]$ is an injective $G_K$-module. But then it is also injective as a $(G_K)_1$-module, see \cite[Rem.~I.4.12, Cor.~I.5.13b)]{Jan}.
So $K\ot_RH^1(G_1,k[\g])=H^i\big((G_K)_1,K[F]\big)=0$ for all $i>0$.

So it now suffices to show that $H^1(G_1,k[\g])$ has no $R$-torsion.
We are going to apply Lemma~\ref{lem.direct_summand}(ii) to the $R$-linear map
$$\varphi:f\mapsto(x\mapsto x\cdot f):k[\g]\to\Hom_k(\g,k[\g])\,.$$
Here the grading of $\Hom_k(\g,k[\g])$ is given by $$\Hom_k(\g,k[\g])^i=\Hom_k(\g,k[\g]^i).$$
As explained in \cite[7.13, 7.14]{Jan3} the conditions of \cite[Prop.~10.1]{Rich} are satisfied under the assumptions (H1)-(H3),
so $k[\g]$ is a free $R$-module. So $\Hom_k(\g,k[\g])$ is also a free $R$-module.
We have $k[\g]/R^+k[\g]=k[\mc N]$, and
$$\ov\varphi=f\mapsto(x\mapsto x\cdot f):k[\mc N]\to\Hom_k(\g,k[\mc N]).$$
By \cite[6.3,6.4]{Jan3}, we have $\min_{x\in\mc N}\dim\g_x=n$ and $\dim\mc N=\dim\g-n$.
So from Lemma~\ref{lem.liealginvs} it is clear that the restriction map $k[\g]\to k[\mc N]$ maps 
the $\g$-invariants of $k[\g]$ onto those of $k[\mc N]$. But $k[\g]^\g=\Ker(\varphi)$ and $k[\mc N]^\g=\Ker(\ov\varphi)$.
So, by Lemma~\ref{lem.direct_summand}(ii), $\Im(\varphi)$ is a direct $R$-module summand of $\Hom_k(\g,k[\g])$.
In particular, $\Hom_k(\g,k[\g])/\Im(\varphi)$ is isomorphic to an $R$-submodule of $\Hom_k(\g,k[\g])$ and therefore $R$-torsion-free.
From \eqref{eq.restrictedcohomology} in Section~\ref{ss.restrictedcohomology} it is clear that $H^1(G_1,k[\g])$
is isomorphic to an $R$-submodule of $\Hom_k(\g,k[\g])/\Im(\varphi)$, so it is also $R$-torsion-free.
\end{proof}

Let $B$ be a Borel subgroup of $G$ containing $T$, let $\b$ be its Lie algebra and let $\u$ be the Lie algebra of the unipotent radical $U$ of $B$.
\begin{thmgl}\label{thm.H^1(B_1,k[b])}
$H^1(B_1,k[\b])=0$.
\end{thmgl}
\begin{proof}
Consider the restriction map $k[\b]^B\to k[\t]$. Under the assumptions (H1)-(H3) $\t$ contains elements which are regular in $\g$.
Furthermore, the set of regular semisimple elements in $\g$ is open in $\g$. So the regular semisimple elements of $\g$
in $\b$ are dense in $\b$. Since the union of the $B$-conjugates of $\t$ is the set of all semisimple elements in $\b$, by \cite[Prop.~11.8]{Bo},
it is also dense in $\b$. This shows that the map $k[\b]^B\to k[\t]$ is injective.
Furthermore, $\Ad(g)(x)-x\in\u$ for all $g\in B$ and $x\in\b$ by \cite[Prop.~3.17]{Bo}, since $DB\subseteq U$. So if we extend $f\in k[\t]$
to a regular function $f$ on $\b$ by $f(x+y)=f(x)$ for all $x\in\t$ and $y\in\u$, then $f\in k[\b]^B$. So the map $k[\b]^B\to k[\t]$ is surjective,
that is, restriction of functions defines an isomorphism $$k[\b]^B\stackrel{\sim}{\to}k[\t]\,.$$
Extend a basis of $\t^*$ to (linear) functions $\xi_1,\ldots,\xi_n$ on $\b$ in the manner indicated above.
Then these functions are algebraically independent generators of $k[\b]^B$, and $k[\b]$ is a free $k[\b]^B$-module.
Clearly, the vanishing ideal of $\u$ in $k[\b]$ is generated by the $\xi_i$.
Furthermore, $\min_{x\in\u}\dim\b_x=n$, see \cite[6.8]{Jan3}. %$\g_x=\Lie(G_x)=\Lie(B_x)\subseteq\b_x$ for $x\in\u$ regular nilpotent
We can now follow the same arguments as in the proof of Theorem~\ref{thm.H^1(G_1,k[g])}.
Just replace $G$, $\g$, $\mc N$, $k[\g]^G$ and the $s_i$ by $B$, $\b$, $\u$, $k[\b]^B$ and the $\xi_i$,
and replace $f_{\rm rs}$ by its restriction to $\b$.
\end{proof}

\begin{remgl}
We have $k[\mc N]=\ind_B^Gk[\u]$. Using \cite[Lemma~II.12.12a)]{Jan} and the arguments from \cite[II.12.2]{Jan}
it follows that $H^1(G_1,k[\mc N])=\ind_B^GH^1(B_1,k[\u])$. From this one can easily deduce examples with
$H^1(G_1,k[\mc N])\ne0$.
\end{remgl}

\section{The cohomology groups $H^1(G_1,k[G])$ and $H^1(B_1,k[B])$}\label{s.H^1(G_1,k[G])}
Assume first that $G=\GL_n$. Put $R=k[\g]^G$ and $R_1=R[\det^{-1}]$. Then, using the fact that the $\g$-action on $k[G]$ is $R_1$-linear,
the Universal Coefficient Theorem and Theorem~\ref{thm.H^1(G_1,k[g])}, we obtain
\begin{align*}
H^1(G_1,k[G])=H^1\big((G_1)_{R_1},k[G]\big)=&R_1\ot_RH^1\big((G_1)_R,k[\g]\big)=\\
&R_1\ot_RH^1(G_1,k[\g])=0\,.
\end{align*}
Similarly, we obtain $H^1(B_1,k[B])=0$.

To prove our result for the case of arbitrary reductive $G$ we assume in this section the following:

{\it There exists a central (see \cite[22.3]{Bo}) surjective morphism %By \cite[Lemma~22.2]{Bo} the centre of $\g$ is schematic central in $G$.
$\varphi:\tilde G\to G$ where $\tilde G$ is a direct product of groups of the following types:
\begin{enumerate}[{\rm (1)}]
\item a simply connected simple algebraic group of type $\ne A$ for which $p$ is good,
\item $\SL_m$ for $p\nmid m$,
\item $\GL_m$,
\item a torus.
\end{enumerate}}

\begin{thmgl}\label{thm.H^1(G_1,k[G])}
$H^1(G_1,k[G])=0$.
\end{thmgl}
\begin{proof}
First we reduce to the case that $G$ is of one of the above four types. Let $\varphi:\tilde G\to G$ be as above.
Then $G$ is the quotient of $\tilde G$ by a (schematic) central diagonalisable closed subgroup scheme $\tilde Z$, see \cite[II.1.18]{Jan}.
Let $N$ be the image of $\tilde G_1$ in $G_1$. Then $N$ is normal in $G_1$ and $G_1/N$ is diagonalisable. %$N$ is reduced, since $\tilde G_1$ is reduced.
So $H^i(G_1,k[G])=H^i(N,k[G])^{G_1/N}$, by \cite[I.6.9(3)]{Jan}.
Furthermore, $H^i(N,k[G])=H^i(\tilde G_1,k[G])$, by \cite[I.6.8(3)]{Jan}, since the kernel of $\tilde G_1\to N$ is central.

The group scheme $\tilde Z$ also acts via the right multiplication action on $k[\tilde G]$ and this action commutes with the conjugation action of $\tilde G$.
So $k[G]=k[\tilde G]^{\tilde Z}$ is a direct $\tilde G$-module summand of $k[\tilde G]$.
So it suffices to show that $H^1(\tilde G_1,k[\tilde G])=0$. By the K\"unneth Theorem we may now assume that $G$ is of one of the above four types.
%One can tensor injective resolutions for the factors together to one for $k[\tilde G]$
%Furthermore, $(M\ot M')^{G\tomes G'}=M^G\ot{M'}^{G'}$ for group schemes $G$ and $G'$ over $k$, a $G$-module $M$ and a $G'$-module $M'$.

For $G$ a torus the assertion is obvious, and for $G=\GL_n$ we have already proved the assertion. Now assume that $G$ is of type (1) or (2).
Then $G$ satisfies (H1)-(H3) and $G$ is simply connected simple.
By \cite[2.15]{St2} the centraliser of a semisimple group element is connected, so when the element is also regular, its
centraliser is a maximal torus. As in the proof of Theorem~\ref{thm.H^1(G_1,k[g])} we are now reduced
to showing that $H^1(G_1,k[G])$ has no torsion over $R:=k[G]^G$.

For this it is enough that $R_{\mf m}\ot_RH^1(G_1,k[G])=H^1(G_1,R_{\mf m}\ot_Rk[G])$ has no torsion over $R_{\mf m}$ for all maximal ideals $\mf m$ of $R$.
By \cite[6.11, 7.16, 8.1]{St} the conditions of \cite[Prop.~10.1]{Rich} are satisfied, so $k[G]$ is a free $R$-module
and $k[G]_\m=R_{\mf m}\ot_Rk[G]$ is a free $R_{\mf m}$-module for all maximal ideals $\mf m$ of $R$.
Furthermore, $k[G]_\m/{\mf m}_{\mf m}k[G]_\m=k[G]/{\mf m}k[G]$ is the coordinate ring of a fiber $F$ of the adjoint quotient map.
We know $F$ is normal of codimension $n$, and a regular orbit closure, so $k[F]^\g=k[F]^p$ by Lemma~\ref{lem.liealginvs}.
By the local version of Lemma~\ref{lem.direct_summand}
the $R_\m$-linear map $\varphi:f\mapsto(x\mapsto x\cdot f):k[G]_\m\to\Hom_k(\g,k[G]_\m)$
we now get that $H^1(G_1,R_{\mf m}\ot_Rk[G])$ has no $R_{\mf m}$-torsion.
\end{proof}

Let $B$ be a Borel subgroup of $G$.
\begin{thmgl}\label{thm.H^1(B_1,k[B])}
$H^1(B_1,k[B])=0$.
\end{thmgl}
\begin{proof}
This follows by modifying the proof of Theorem~\ref{thm.H^1(G_1,k[G])} in the same way as the proof of Theorem~\ref{thm.H^1(G_1,k[g])}
was modified to obtain the proof of Theorem~\ref{thm.H^1(B_1,k[b])}.
\end{proof}

\begin{remsgl}
1.\ One can also prove Theorem~\ref{thm.H^1(B_1,k[B])} assuming (H1)-(H3). The point is that it is obvious that restriction of functions
always defines an isomorphism $k[B]^B\stackrel{\sim}{\to} k[T]$.\\ % Since (H1)-(H3) pass down to Levi's we also know that
%all fibers of the adjoint quotient contain a regular element whose infinitesimal centraliser (in \b) also has dimension n.
2.\ We briefly discuss the $B$-cohomology of $k[B]$ and $k[\b]$. From \cite[Thm 1.13, Thm 1.7(a)(ii)]{vdK} it is immediate that $H^i(B,k[B])=0$
for all $i>0$. % Or see also Thm 3.2.6 in his Lecture notes on Frobenius Splitting
Now assume that there exists a central surjective morphism $\varphi:\tilde G\to G$ where $\tilde G$ is a direct product of groups of the types (1)-(4) mentioned before,
except that for type (2) we drop the condition on $p$. Then we deduce from the arguments from the proof of \cite[Prop.~4.4]{AJ} that $H^i(B,k[\b])=0$ for all $i>0$ as follows.
First we reduce as in the proof of Theorem~\ref{thm.H^1(G_1,k[G])} to the case that $G$ is simple of type (1) or (2) and then we deal with type (2) as in \cite{AJ}.
Now assume $G$ is of type (1) and let $I$ be the vanishing ideal of $B$ in $k[G]$. As in \cite{AJ} write
\begin{equation}\label{eq.m}
\m=M\oplus\m^2
\end{equation}
where
$\m$ is the vanishing ideal in $k[G]$ of the unit element of $G$ and $M\cong\g^*$ as $G$-modules. It suffices to show that $I=I\cap M+I\cap\m^2$,
since then we get a decomposition analogous to \eqref{eq.m} for $k[B]$ and we can finish as in \cite{AJ}. Let $f\in I$. Then the $M$-component of $f$ correspond to $df\in\g^*$
which vanishes on $\b$. This means it corresponds under the trace form of the chosen representation $\rho:G\to V$ (the adjoint representation for exceptional types)
to an element $x\in\u$. So the $M$-component of $f$ is $g\mapsto\tr\big(\rho(g)d\rho(x)\big)$ which vanishes on $B$. But then the $\m^2$-component of $f$ must also vanish on $B$.
\end{remsgl}
\begin{comment}
One could prove that $H^1(B,k[\b])=0$ following the proof of Theorem~\ref{thm.H^1(G_1,k[g])}. Just replace the map $\varphi$ by the map
$$f\mapsto(g\mapsto g\cdot f-f):k[\b]\to\Mor_{k-{\rm func}}(B,k[\b])\,.$$
Similarly one could prove that $H^1(B,k[B])=0$ following the proof of Theorem~\ref{thm.H^1(G_1,k[G])}
\end{comment}

\section{The cohomology groups for the higher Frobenius kernels}\label{s.higherFrobenius}
In this section we will generalise the results from the previous two sections to all Frobenius kernels $G_r$, $r\ge 1$.

\begin{lemgl}\label{lem.liealginvs2}
Let $G$ be a linear algebraic group over $k$ acting on a normal affine variety $X$ over $k$.
If $\max_{x\in X}\codim_\g\g_x=\dim X$, then $k[X]^{G_r}=k[X]^{p^r}$ for all integers $r\ge1$.
\end{lemgl}
\begin{proof}
Since $\codim_\g\g_x\le\codim_GG_x\le\dim(X)$ and $\max_{x\in X}\codim_\g\g_x=\dim X$ we must have
that for $x\in X$ with $\codim_\g\g_x=\dim X$ the schematic centraliser of $x$ in $G$ is reduced.
So $(G_r)_x=(G_x)_r$ and $$(G_r:(G_r)_x):=\dim(k[G_r])/\dim(k[(G_r)_x])=p^{r\dim(G)}/p^{r\dim(G_x)}=p^{r\dim(X)}.$$
By \cite[Thm.~2.1(5)]{Skry} we get $[k(X):k(X)^\g]=p^{r\dim(X)}$.
By \cite[Cor.~3 to Thm.~V.16.6.4]{Bou0} and the tower law %take K=k and L=k(X); if $[L:L^p]=p^{\dim(X)$, then also $[L^{p^{r-1}}:L^{p^r}]=p^{\dim(X)$
we have $[k(X):k(X)^{p^r}]=p^{r\dim(X)}$. So $k(X)^\g=k(X)^{p^r}$, since we always have $\supseteq$.
Clearly, $k(X)^{p^r}={\rm Frac}(k[X]^{p^r})$, $k(X)^\g={\rm Frac}(k[X]^\g)$ and
$k[X]^\g$ is integral over $k[X]^{p^r}$. Since $X$ is normal variety, $k[X]^{p^r}\cong k[X]$ is a normal ring.
It follows that $k[X]^\g=k[X]^{p^r}$.
\end{proof}

\begin{thmgl}\label{thm.H^1(G_r,k[g])}
Let $r$ be an integer $\ge1$.
\begin{enumerate}[{\rm(i)}]
\item Under the assumptions of Section~\ref{s.H^1(G_1,k[g])} we have
$$H^1(G_r,k[\g])=0\text{ and }H^1(B_r,k[\b])=0.$$
\item Under the assumptions of Section~\ref{s.H^1(G_1,k[G])} we have
$$H^1(G_r,k[G])=0\text{ and }H^1(B_r,k[B])=0.$$
\end{enumerate}
\end{thmgl}
\begin{proof}
(i) Let $(H,M)$ be the group and module in question, i.e. $(G,k[\g])$ or $(B,k[\b])$.
Put $R=k[\h]^H$. Let $\varphi$ be the first map in the Hochschild complex of the $H_r$-module $M$, see \cite[I.4.14]{Jan}:
$$\varphi=f\mapsto(\Delta_M(f)-1\ot f):M\to k[H_r]\ot M.$$
Then the induced map $\ov\varphi:M/R^+M\to k[H_r]\ot(M/R^+M)$ is the first map in the Hochschild complex of the $H_r$-module $M/R^+M$ which is $k[\mc N]$ or $k[\u]$.
So $\Ker(\varphi)=M^{H_r}$ and $\Ker(\ov\varphi)=(M/R^+M)^{H_r}$.
Now the proof is the same as that of the corresponding result in Section~\ref{s.H^1(G_1,k[g])}, except that we work with the above map $\varphi$
and instead of Lemma~\ref{lem.liealginvs} we apply Lemma~\ref{lem.liealginvs2}.\\
(ii).\ Let $(H,M)$ be the group and module in question, i.e. $(G,k[G])$ or $(B,k[B])$.
As in the proof of the corresponding result in Section~\ref{s.H^1(G_1,k[G])} we reduce to the case that $G$ is simple of type (1) or (2).
Put $R=k[H]^H$. Fix a maximal ideal $\m$ of $R$. Let $\varphi$ be the first map in the Hochschild complex of the $H_r$-module $M_\m$.
Then the induced map $\ov\varphi:M_\m/\m_\m M_\m\to k[H_r]\ot M_\m/\m_\m M_\m$ is the first map in the Hochschild complex of
the $H_r$-module $M_\m/\m_\m M_\m=M/\m M$ which is the coordinate ring of the fiber of $H\to H/\nts/H$ over the point $\m$.
So $\Ker(\varphi)=(M_\m)^{H_r}$ and $\Ker(\ov\varphi)=(M/\m M)^{H_r}$.
Now the proof is the same as that of the corresponding result in Section~\ref{s.H^1(G_1,k[G])}, except that we work with the above map $\varphi$
and instead of Lemma~\ref{lem.liealginvs} we apply Lemma~\ref{lem.liealginvs2}.
\end{proof}

\begin{comment}
Our proof doesn't work for the higher cohomology groups, because the condition of Lemma~\ref{lem.direct_summand}(ii) doesn't hold
for the later maps in the Hochschild complex.
\end{comment}

\begin{remgl}
For $G$ classical with natural module $V=k^n$ we consider the cohomology groups $H^1(G_r,S^iV)$ and $H^1(G_r,S^i(V^*))$.

Results about these modules can mostly easily be deduced from results on induced modules in the literature.
For induced modules one can reduce to $B_r$-cohomology using the following result of Andersen-Jantzen for general $G$.
Let $B$ be a Borel subgroup of $G$ with unipotent radical $U$ and let $T$ be a maximal torus of $B$.
For $\lambda\in X(T)$, the character group of $T$, we denote by $\nabla(\lambda)$, the $G$-module induced from the 1-dimensional $B$-module given by $\lambda$.
We call the roots of $T$ in the opposite Borel subgroup $B^+$ positive.
By \cite[II.12.2]{Jan} we have for $\lambda$ dominant
\vspace{-2mm}
\begin{equation}\label{eq.ind}
H^1(G_r,\nabla(\lambda))^{[-r]}\cong\ind_B^G(H^1(B_r,\lambda)^{[-r]}).
\end{equation}

Below we will always take $\lambda=\varpi_1$ the first non-constant diagonal matrix coordinate.
First take $G=\GL_n$. Let $B$ and $T$ be the lower triangular matrices and the diagonal matrices.
Then the character group $X(T)$ of $T$ identifies with $\mathbb Z^n$. Let $\ve_1$ be the first standard basis element of $X(T)$, i.e.
the character $D\mapsto D_{ii}$. Then $S^iV=\nabla(i\ve_1)$ and $S^i(V^*)=\nabla(-i\ve_n)$.
Replacing $\u^{*[s]}$ by $\lambda\ot\u^{*[s]}$ for $\lambda=i\ve_1$ or $\lambda=-i\ve_n$ in the proof of \cite[Lem.~II.12.1]{Jan}
and using \eqref{eq.ind} we obtain $H^1(G_r,S^iV)=H^1(G_r,S^i(V^*))=0$.

Now take $G=\SL_n$. Then $S^iV=\nabla(i\varpi_1)$ and $S^i(V^*)=\nabla(i\varpi_{n-1})$, 
where $\varpi_j$ denotes the $j$-th fundamental dominant weight. From \cite[Cor.~3.2(a)]{BN} we easily deduce
that $H^1(G_r,S^iV)\ne0$ if and only if $H^1(G_r,S^i(V^*))\ne0$ if and only if
$n=2$ and $p^r\mid i+2p^s$ for some $s\in\{0,\ldots,r-1\}$, or $n=3$, $p=2$ and $2^r\mid i-2^{r-1}$.

For $G=\Sp_n$, $n\ge4$ even, we deduce using $S^i(V)=\nabla(i\varpi_1)$ and \cite[Cor.~3.2(a)]{BN} that $H^1(G_r,S^iV)\ne0$ if and only if $p=2$ and $i$ is odd. %$\Sp_2=\SL_2$

Now let $G$ be the special orthogonal group $\SO_n$, $n\ge4$, as defined in \cite[Ex.~7.4.7(3),(4),(6),(7)]{Spr} (when $p=2$ this is an abuse of notation). 
%$\SO_2$ would be a torus, $\SO_3=\PGL_2$ and the $\SO_3$-module $V^*$ corresponds then the induced $\PGL_2$-module $\nabla(\alpha)=\pgl_2$.
Note that $V\cong V^*$ unless $n$ is odd and $p=2$.
Although the simply connected cover $\tilde G\to G$ need not be separable, it still follows from \cite[I.6.8(3), I.6.9(3)]{Jan} that
$H^1(G_r,M)=H^1(\tilde G_r,M)^{T_r}$ for any $G$-module $M$, and $H^1(B_r,M)=H^1(\tilde B_r,M)^{T_r}$ for any $B$-module $M$.
So one has to pick out the weight spaces of the weights in $p^rX(T)\subseteq p^rX(\tilde T)$. %For $n\ge 8$ we have in fact $H^1(\tilde G_r,\nabla(i\varpi_1))=0$ by \cite[Cor.~3.2(a)]{BN}
For $n\ge8$ it follows from \cite[Cor.~3.2(a)]{BN} that $H^1(\tilde G_r,\nabla(i\varpi_1))=0$ for all $i\ge0$. For general $n\ge4$ we proceed as follows.
From \cite[Sect.~2.5-7]{BN} we deduce that all weights of $H^1(B_r,i\varpi_1)$ are of the form $i\varpi_1+p^s\alpha$ for some $s\in\{0,\ldots,r-1\}$ %I think one can actually take s=r-1
and some $\alpha$ simple or ``long" (i.e. there is a shorter root). Since such weights don't occur in $p^rX(T)$ for $\SO_n$, $n\ge4$, we get that
$H^1(B_r,i\varpi_1)=0$, and therefore by \eqref{eq.ind} $H^1(G_r,\nabla(i\varpi_1))=0$ for all $i\ge0$.
By \cite[II.2.17,18]{Jan} $S^i(V^*)$ has a filtration with sections $\nabla(i\varpi_1), \nabla((i-2)\varpi_1),\ldots$.
So $H^1(G_r,S^i(V^*))=0$ for all $i\ge0$.

The fact that the weights of $H^1(B_r,i\varpi_1)$ have the form stated above can been seen more directly as follows.
First one observes that $1$-cocyles in the Hochschild complex of a $U_r$-module $M$ can be seen as the linear maps $D:\Dist^+(U_r)\to M$ with $D(ab)=aD(b)$ for all $a\in\Dist(U_r)$ and $b\in\Dist^+(U_r)$.
Here $\Dist^+(U_r)$ denotes the distributions without constant term, i.e. the distributions $a$ with $a(1)=0$.
Then one shows that, outside type $G_2$, $\Dist(U_r)$ is generated by the $\Dist(U_{-\alpha,r})$ with $\alpha$ simple or long.\footnote{
If $p$ is not special in the sense of \cite{Jan1}, then (also in type $G_2$) $\Dist(U_r)$ is generated by the $\Dist(U_{-\alpha,r})$ with $\alpha$ simple.}
%We can immediately reduce to the case that the root system is irreducible.
%We show by induction on the height that the $\Dist(U_{-\gamma})$, $\gamma$ short are in the subalgebra generated by the $\Dist(U_{-\alpha})$, $\alpha$ simple 
%for $\gamma>0$ short write $\gamma=\alpha+\beta$ with $\alpha,\beta$ of smaller height. Consider the $\alpha$-string through $\beta$.
%Let $r$ be the number of times you have to add $\alpha$ to the  first root in the string to get $\gamma$.
%Then $[x_{\alpha},x_{\beta}]=\pm rx_{\gamma}$. If the string has length $1$, then $r=1$.
%If the string has length $2$ (the only possibility, since we are outside type $G_2$),
%then the final root is long, so again $r=1$. Now finish as in the old section 5.
%It is easy to see that in type $G_2$ the simple and long root vectors do not generate $$
It follows that $H^1(U_r,M)$ is a subquotient of $M\otimes\bigoplus_{\alpha,\, 0\le s<r}\u_{-\alpha}^{*[s]}$, the $\alpha$ simple or long.
Now use that, for $M$ a $B_r$-module, $H^1(B_r,M)=H^1(U_r,M)^{T_r}$.
\end{remgl}

\noindent{\it Acknowledgement}. I would like to thank H.~H.~Andersen and J.~C.~Jantzen for helpful email discussions.

\bigskip

{\sc\noindent School of Mathematics,\\
University of Leeds, LS2 9JT, Leeds, UK.\\
{\tt R.H.Tange@leeds.ac.uk}

\end{document}